\documentclass[12pt, reqno, a4paper]{amsart}
\usepackage[margin=1.2in]{geometry}
\numberwithin{equation}{section}
\usepackage{amssymb,amsfonts,amsthm}
\usepackage[utf8]{inputenc}
\usepackage{listings}
\usepackage{bm}
\usepackage[hyperpageref]{backref}
\usepackage{esint}
\usepackage{color}
\usepackage[bookmarks]{hyperref}
\usepackage{siunitx}
\usepackage{tikz}
\usetikzlibrary{positioning,arrows.meta}
\usepackage{bigints}
\usepackage{tikz}
\usepackage{enumitem}
\usepackage{subcaption}
\usepackage{hyperref}
\usepackage{mathtools}
\usepackage{booktabs}

\newcommand{\Cc}{C^1_c}
\newcommand{\Om}{\Omega}
\DeclareMathOperator{\supp}{supp}

\allowdisplaybreaks[4]

\newtheoremstyle{myremark}{10pt}{10pt}{}{}{\bfseries}{.}{.5em}{}

\newtheorem{theorem}{Theorem}[section]
\newtheorem{corollary}[theorem]{Corollary}

\newtheorem{lemma}[theorem]{Lemma}
\newtheorem{proposition}[theorem]{Proposition}

\theoremstyle{definition}
\newtheorem{definition}[theorem]{Definition}
\newtheorem{remark}{Remark}

\usepackage{hyperref}

 \newcommand{\R}{\mathbb{R}}

 \newcommand{\al}{\alpha}

\DeclareMathOperator{\SO}{SO}

\DeclareMathOperator{\dist}{dist}

\allowdisplaybreaks[4]

\begin{document}

\title[Hardy--Sobolev--Maz'ya inequalities]{Hardy--Sobolev--Maz'ya inequalities for $\boldsymbol{1<p<2}$}

\author[A{.} Banerjee, M{.} Kijaczko, and V{.} Sahu]{Avas Banerjee, Michał Kijaczko, and Vivek Sahu}

\address[A{.}B{.}]{Theoretical Statistics and Mathematics Unit,
Indian Statistical Institute, Delhi Centre, S.J. Sansanwal Marg, New Delhi, Delhi 110016, India}
\email{avas24r@isid.ac.in}

\address[M{.}K{.}]{Faculty of Pure and Applied Mathematics, Wrocław University of Science and Technology, Wybrzeże Wyspiańskiego 27, 50-370 Wrocław, Poland}
\email{michal.kijaczko@pwr.edu.pl}

\address[V{.}S{.}]{Theoretical Statistics and Mathematics Unit,
Indian Statistical Institute, Delhi Centre, S.J. Sansanwal Marg, New Delhi, Delhi 110016, India}
\email{vivek@isid.ac.in,  viiveksahu@gmail.com}



\date{}

\dedicatory{}

\begin{abstract}
We establish Hardy–Sobolev–Maz’ya inequalities in the subquadratic range $1<p<2$ for singularities supported on affine subspaces and for arbitrary domains. We first prove a Hardy–Sobolev–Maz’ya inequality for singularities supported on affine subspaces of arbitrary codimension $1\leq k<N$, showing that the deficit in the sharp Hardy inequality controls a suitable weighted Sobolev norm. Combining the same convexity estimate with the one-dimensional approach of Frank and Loss in \cite{FrankLoss2012}, we also establish Hardy--Sobolev--Maz'ya inequalities on arbitrary domains with the mean distance to the boundary and thereby solve a problem posed in \cite{FrankLoss2012}. These results complete the previously unresolved subquadratic range $1<p<2$ and, together with the known results for $p\geq 2$, provide a unified Hardy--Sobolev--Maz'ya theory for all $1<p<N$ associated with plane singularities and general domains.
\end{abstract}

\maketitle


\section{Introduction}

Hardy inequalities constitute one of the cornerstones of modern analysis and have profound connections with mathematical physics, partial differential equations, spectral theory, and the calculus of variations. Their origins can be traced back to the classical one-dimensional inequality established by Hardy \cite{Hardy1925}, which states that for every $u\in C^{\infty}_c\left((0,\infty)\right)$ and $1<p<\infty$,
\begin{equation*}
    \int_0^\infty |u'(t)|^p\,dt
\geq
\left(\frac{p-1}{p}\right)^p
\int_0^\infty \frac{|u(t)|^p}{t^p}\,dt .
\end{equation*}
This inequality provides a fundamental relationship between a function and its derivative and has inspired a vast body of research devoted to multidimensional extensions, sharp constants, and remainder terms.

Among the most important higher-dimensional counterparts is the Hardy inequality with a point singularity. For $1<p<N$, it asserts
\begin{equation*}
    \int_{\mathbb{R}^{N}} |\nabla u(x)|^{p} \,dx
 \geq
\left(\frac{N-p}{p}\right)^{p}
\int_{\mathbb{R}^{N}}
\frac{|u(x)|^p}{|x|^p}\,dx,
\quad \forall \,
u\in C^{\infty}_{c}(\mathbb{R}^{N}),
\end{equation*}
where the constant $\left(\frac{N-p}{p}\right)^{p}$ is optimal. This inequality has become a fundamental tool in the study of elliptic and parabolic equations, spectral estimates, singular potentials, and various problems in geometric analysis.

\smallskip

A natural generalization arises when the singularity is distributed along a lower dimensional subspace rather than concentrated at a single point. Let $x \in \mathbb{R}^{N}$, we write $x= (x_{k}, x_{N-k}) \in \mathbb{R}^{k} \times \mathbb{R}^{N-k}$, where $1 \leq k < N$, and consider the affine subspace
\begin{equation*}
    K:= \{ x = (x_{k}, x_{N-k}) \in \mathbb{R}^{k} \times \mathbb{R}^{N-k} : x_{k} = 0 \},
\end{equation*}
which has codimension $k$. In the limiting case $k=N$, the set $K$ reduces to the origin, that is, $K = \{ 0\}$. The corresponding Hardy inequality with singularity on a plane takes the form
\begin{equation}\label{Hardy inequality on submanifold}
     \int_{\mathbb{R}^{N}} |\nabla u(x)|^{p} \,dx
 \geq
\left|\frac{k-p}{p}\right|^{p}
\int_{\mathbb{R}^{N}}
\frac{|u(x)|^p}{|x_{k}|^p}\,dx.
\end{equation}
This inequality holds for all $u\in C^{\infty}_{c}(\mathbb{R}^{N})$ when $1<p<k$, and for all $u\in C^{\infty}_{c}(\mathbb{R}^{N} \setminus K)$ when $p>k$. Moreover, the constant $\left|\frac{k-p}{p}\right|^{p}$ is sharp. Hardy inequalities associated with singular subspaces have attracted considerable attention over the past decades, owing both to their intrinsic analytical interest and to their applications in the study of partial differential equations with anisotropic or singular structures.

\smallskip
\subsection{Hardy--Sobolev--Maz'ya inequalities on $\R^N$}
A significant breakthrough in the theory of Hardy inequalities was achieved by Maz'ya \cite{MazyaBook}, who observed that Hardy inequalities admit remarkable Sobolev-type improvements. More precisely, the deficit obtained by subtracting the Hardy term from the Dirichlet energy often controls a Sobolev norm of the underlying function. Such inequalities are now commonly referred to as Hardy--Sobolev--Maz'ya (HSM) inequalities and have become a fundamental tool in the analysis of singular variational problems and partial differential equations.

One of the pioneering results in this direction was established by Maz'ya in \cite[Corollary 3, Section 2.1.7]{MazyaBook}. He proved that for $1 \leq k < N$, $p=2$, and $2 < q \leq \frac{2N}{N-2}$, there exists a constant $C>0$ such that
\begin{equation*}
     \int_{\mathbb{R}^{N}} |\nabla u(x)|^{2} \,dx
-
\left(\frac{k-2}{2}\right)^{2}
\int_{\mathbb{R}^{N}}
\frac{|u(x)|^{2}}{|x_{k}|^2}\,dx \geq C \left( \int_{\mathbb{R}^{N}} \frac{|u(x)|^{q}}{|x_{k}|^{N+ (2-N)\frac{q}{2}}}  \right)^{\frac{2}{q}},
\end{equation*}
for all admissible functions $u$.

In the critical case $q=\frac{2N}{N-2}$, the weight in the right-hand side disappears, and the above inequality reduces to
\begin{equation*}
    \int_{\mathbb{R}^{N}} |\nabla u(x)|^{2} \,dx
-
\left(\frac{k-2}{2}\right)^{2}
\int_{\mathbb{R}^{N}}
\frac{|u(x)|^{2}}{|x_{k}|^2}\,dx \geq C \left( \int_{\mathbb{R}^{N}} |u(x)|^{q} \, dx\right)^{\frac{2}{q}}.
\end{equation*}

Maz'ya further observed that these HSM inequalities fail in the limiting case $k=N$, corresponding to a point singularity. This phenomenon highlights a fundamental distinction between singularities supported on lower-dimensional subspaces and those concentrated at a single point.

The HSM inequalities were subsequently extended to the nonlinear setting by Barbatis, Filippas, and Tertikas \cite{BarbatisTAMS2004}. Specifically, they proved that for $1\leq k<N$, $2\leq p<N$, and $p<q\leq \frac{Np}{N-p}$, there exists a constant $C>0$ such that
\begin{equation}\label{HSM}
     \int_{\mathbb{R}^{N}} |\nabla u(x)|^{p} \,dx
-
\left| \frac{k-p}{p}\right|^{p}
\int_{\mathbb{R}^{N}}
\frac{|u(x)|^{p}}{|x_{k}|^p}\,dx \geq C \left( \int_{\mathbb{R}^{N}} \frac{|u(x)|^{q}}{|x_{k}|^{ N+(p-N)\frac{q}{p}}}  \right)^{\frac{p}{q}},
\end{equation}
for every $u\in C^{\infty}_{c}(\mathbb{R}^{N}\setminus K)$.

This result extends Maz'ya's HSM inequality from the quadratic
case $p=2$ to the nonlinear regime $p\geq 2$. It shows that the Hardy
deficit controls a weighted Sobolev norm of the function, providing a
quantitative improvement of the classical Hardy inequality in the
presence of singular subspaces. The value of the sharp constant $C$ in \eqref{HSM} remains unknown, besides one case: for the half-space, $p=2$, $N=3$, and $q=6$, Benguria, Frank and Loss \cite{MR2424899} proved a remarkable fact that it coincides with sharp Sobolev constant. For a related result, it is also worth to mention the work of Tertikas and Tintarev \cite{MR2317783}, dealing with the existence of minimizer for \eqref{HSM}, when $p=2$. For more information on HSM and related inequalities involving Hardy deficit, we also refer to a review article of Psaradakis \cite{Psaradakis} and to \cite{FilippasPsaradakis}, where Hardy--Morrey inequalities for $p>N$ were studied.

However, the validity of HSM inequalities in the subquadratic range
$1<p<2$ has remained an open problem. The difficulty here is not merely
technical. The geometry of the Hardy singularity interacts with the
degeneracy of the map $\R^N\ni X\mapsto |X|^p$ near the origin. In particular,
when $1<p<2$, its Hessian becomes singular at $X=0$, preventing a direct
use of the second-order convexity arguments that are effective in the
range $p\geq2$.

\smallskip

The main objective of the present paper is to overcome this obstruction
and establish HSM inequalities in the full subquadratic range. Our
approach combines a quantitative remainder estimate for the convexity
of $|X|^p$ with a ground-state transformation and a weighted endpoint
$L^1$ Sobolev inequality. This allows us to recover a weighted Sobolev
norm from the deficit in the sharp Hardy inequality even when
$1<p<2$.

More precisely, we prove an HSM inequality for singularities supported
on affine subspaces of arbitrary codimension $1\leq k<N$. For
$p<q\leq p^*=Np/(N-p)$, the Hardy deficit controls a weighted
$L^q$-norm whose weight is determined by both the dimension $N$ and
the codimension $k$ of the singular set. Thus, our result extends the
known theory from $p\geq2$ to $1<p<2$ and completes the HSM theory for
affine singularities throughout the range $1<p<N$.

\smallskip

The codimension-one case $k=1$ is of particular interest. In this
case the singular set is a hyperplane, and restriction to either
component of its complement gives the corresponding inequality in a
half-space. This provides a direct connection between the general
affine-singularity result and the classical Hardy inequality on the
half-space.

\smallskip

Our main result establishes a Hardy--Sobolev--Maz'ya inequality
associated with the singular subspace $K$ for the range $1<p<2$.

\begin{theorem}\label{Theorem: HSM Plane Singularity}
 Let $N \geq 2$, $1<p<2$, and  $p<q \leq p^{*} = \frac{Np}{N-p}$. For any $x\in \mathbb{R}^{N}$, we write $x= (x_{k}, x_{N-k}) \in \mathbb{R}^{k} \times \mathbb{R}^{N-k}$, where $1 \leq k <N$. Then for all $u \in C^{1}_{c}(\mathbb{R}^{N})$ when $1<p<k$, and for all $u \in C^{1}_{c}(\mathbb{R}^{N} \setminus K)$ when $p > k$, there exists a constant $C=C(N,p,q,k)>0$ such that 
 \begin{equation*}
\int_{\mathbb{R}^{N}} |\nabla u(x)|^{p} \, dx - \left| \frac{k-p}{p} \right|^{p} \int_{\mathbb{R}^{N}} \frac{|u(x)|^{p}}{|x_{k}|^{p}} \, dx \geq C \left(\int_{\mathbb{R}^{N}}  \frac{|u(x)|^{q}}{|x_{k}|^{N+(p-N)\frac{q}{p}}} \, dx \right)^{\frac{p}{q}} .
 \end{equation*}
 In particular, at the critical exponent $q=p^*$, the weight on the
right-hand side disappears and
\begin{equation*}
\int_{\mathbb{R}^{N}} |\nabla u(x)|^{p} \, dx - \left| \frac{k-p}{p} \right|^{p} \int_{\mathbb{R}^{N}} \frac{|u(x)|^{p}}{|x_{k}|^{p}} \, dx \geq C \left(\int_{\mathbb{R}^{N}} |u(x)|^{p^{*}} \, dx \right)^{\frac{p}{p^{*}}} .
 \end{equation*}
\end{theorem}

The above theorem completes the remaining open range $1<p<2$ in the HSM inequality for the half-space. 

\begin{corollary}[Half-space HSM inequality]
Let $N\geq 2$, $1<p<2$, and $p<q\leq p^*=\frac{Np}{N-p}$. For any $x \in \mathbb{R}^{N}_{+}$, we write $x= (x_{N-1}, x_{N}) \in \mathbb{R}^{N-1} \times \mathbb{R}_{+}$ Then there exists a constant $C=C(N,p,q)>0$ such that, for every
$u\in C_c^1(\mathbb{R}^N_+)$,
\begin{align*}
\int_{\mathbb{R}^N_+}|\nabla u(x)|^p\,dx
-\left(\frac{p-1}{p}\right)^p
\int_{\mathbb{R}^N_+}\frac{|u(x)|^p}{x_N^p}\,dx \geq
C\left(
\int_{\mathbb{R}^N_+}
\frac{|u(x)|^q}
{x_N^{\,N+(p-N)\frac{q}{p}}}\,dx
\right)^{\frac{p}{q}}.
\end{align*}
In particular, at the critical exponent $q=p^*$, the weight on the
right-hand side disappears and
\begin{align*}
\int_{\mathbb{R}^N_+}|\nabla u(x)|^p\,dx
-\left(\frac{p-1}{p}\right)^p
\int_{\mathbb{R}^N_+}\frac{|u(x)|^p}{x_N^p}\,dx \geq
C\left(
\int_{\mathbb{R}^N_+}|u(x)|^{p^*}\,dx
\right)^{\frac{p}{p^*}}.
\end{align*}
\end{corollary}

The proof of Theorem \ref{Theorem: HSM Plane Singularity} is based on two main ingredients. First, in the subquadratic range $1<p<2$, we apply a quantitative remainder estimate for the Hardy inequality. This is obtained using a pointwise convexity estimate for $|X|^p$ together with the ground-state transformation. Second, we prove a weighted Sobolev inequality and use it to estimate the remainder term in terms of the weighted $L^q$ norm appearing in Theorem \ref{Theorem: HSM Plane Singularity}. A cut-off argument is used in the case $p<k$ to handle functions which may not vanish near the singular set.

\medskip
\begin{remark}
For the sake of completeness, let us notice that the fractional version of the Hardy--Sobolev--Maz'ya inequality, that is
$$
\int_{\mathbb{R}^N}\int_{\mathbb{R}^N}\frac{|u(x)-u(y)|^p}{|x-y|^{N+sp}}\,dy\,dx-\mathcal{C}_{N,s,p,k}\int_{\mathbb{R}^N}\frac{|u(x)|^p}{|x_k|^{sp}}\,dx\ge C\left(\int_{\mathbb{R}^N}\frac{|u(x)|^q}{|x_k|^{N+(sp-N)\frac{q}{p}}}\,dx\right)^{p/q},
$$
where $\mathcal{C}_{N,s,p,k}$ is the sharp constant in a corresponding fractional Hardy inequality and $0<s<1$, $sp<N$, $p<q\le\tfrac{Np}{N-sp}$, has been recently established for $1<p<2$ by Dyda and Kijaczko \cite{MR4708667} for the case of the half-space and by Kijaczko and Sahu \cite{VivekCCM2026} in the whole generality of singularities on flat submanifolds. Somewhat unexpectedly, the fractional case was proven earlier than the local case derived in this article. For other papers dealing with fractional Hardy--Sobolev--Maz'ya inequalities, we refer to \cite{MR2910984,MR3803664,MR2823046}.
\end{remark}

\subsection{Hardy–Sobolev–Maz'ya inequalities on general domains}
HSM inequalities can also be considered on arbitrary domains $\Omega\subset\R^N$. To describe this, we first define the so-called pseudodistance (or mean distance), originally introduced by Davies \cite{MR757118} for $p=2$ and generalized to all $p>1$ by Tidblom \cite{MR2052402}, given by the formula
$$
D_{\Omega,p}(x):=\left(\frac{\sqrt{\pi}\,\Gamma\left(\frac{N+p}{2}\right)}{\Gamma\left(\frac{p+1}{2}\right)\Gamma\left(\frac{N}{2}\right)}\left|\mathbb{S}^{N-1}\right|^{-1}\int_{\mathbb{S}^{N-1}}d_{e}(x)^{-p}d\sigma(e)\right)^{-\frac{1}{p}},
$$
where, for $e\in\mathbb{S}^{N-1}$, 
\begin{equation}\label{def: de}
d_{e}(x):=\inf\left\{|t|\colon x+te\in\Om^c\right\}.
\end{equation}
For convex $\Om$, we always have $D_{\Om,p}(x)\le\text{dist}(x,\Om^c)$. For any $\Om\subsetneq\R^N$ and $p>1$, there is a Hardy inequality
\begin{equation}\label{Hardy domain}
\int_{\Om}|\nabla u(x)|^p\,dx\ge\left(\frac{p-1}{p}\right)^p\int_{\Om}\frac{|u(x)|^p}{D_{\Om,p}(x)^p}\,dx,\quad u\in C_c^1(\Om).
\end{equation}
It turns out that \eqref{Hardy domain} enjoys a Sobolev-type improvement. Frank and Loss \cite{FrankLoss2012} established a HSM inequality for arbitrary domains and $2\le p<N$. They proved that there exists a constant $C=C(p,N)>0$, independent on $\Om$, such that with $q=\tfrac{Np}{N-p}$,
\begin{equation}\label{HSM bounded domains}
\int_{\Om}|\nabla u(x)|^p\,dx-\left(\frac{p-1}{p}\right)^p\int_{\Om}\frac{|u(x)|^p}{D_{\Om,p}(x)^p}\,dx\ge C\left(\int_{\Om}|u(x)|^q\,dx\right)^{p/q},
\end{equation}
for all $u\in C_c^1(\Om)$. As in the case of the whole space, the question of whether inequality \eqref{HSM bounded domains} holds for $1<p<2$ remained an open problem. Our second main result gives an affirmative answer to this question and reads as follows. 

\begin{theorem}\label{HSM domains p<2}
Let $\Om\subset\R^N$ be any proper domain. Let $1<p<2$, $N\ge 2$, $q=\tfrac{Np}{N-p}$, and  $u\in C_c^1(\Om)$. There exists a constant $C=C(p,N)>0$ such that 
\begin{equation}\label{HSM bounded domains 1<p<2}
\int_{\Om}|\nabla u(x)|^p\,dx-\left(\frac{p-1}{p}\right)^p\int_{\Om}\frac{|u(x)|^p}{D_{\Om,p}(x)^p}\,dx\ge C\left(\int_{\Om}|u(x)|^q\,dx\right)^{p/q}.
\end{equation}
In particular, if $\Om$ is convex, then also 
$$
\int_{\Om}|\nabla u(x)|^p\,dx-\left(\frac{p-1}{p}\right)^p\int_{\Om}\frac{|u(x)|^p}{\dist(x,\Om^c)^p}\,dx\ge C\left(\int_{\Om}|u(x)|^q\,dx\right)^{p/q}.
$$
\end{theorem}
Notice that, since $p<2$, Theorem \ref{HSM domains p<2} also covers the planar case $N=2$, which is excluded in \eqref{HSM bounded domains}.

\smallskip
\noindent\textbf{Organization of the article.}
The paper is organized as follows. In Section \ref{Section2}, we recall Hardy inequality with a quantitative remainder in the range $1<p<2$. In Section \ref{Section3}, we prove the weighted $L^1$ Sobolev inequality and establish the estimates needed to control the remainder term. We also prove the cut-off result near the singular set. In Section \ref{Section4}, we combine these estimates to prove Theorem \ref{Theorem: HSM Plane Singularity}. Finally, in Section \ref{Section5}, we prove Theorem \ref{HSM domains p<2}.

\section{A Hardy inequality with a quantitative remainder}\label{Section2}



We shall require the following quantitative lower bound for the
increment of the map $X\mapsto |X|^p$ in the range $1<p<2$. The estimate
captures a second-order contribution to the convexity of this map, with
a remainder that is suitable for the degeneracy of the Hessian near the
origin. Such estimates (with various constants) are known in the literature, see \cite{MR1007505,MR2400106}. An estimate of this type with a constant $\tfrac{p(p-1)}{2}$ was obtained in \cite{Debdip2026}.

\begin{lemma}\label{lem:elem}
Let $N\ge1$, $1<p\le2$, $a,b\in\R^N$, and set $|a|^{p-2}a:=0$ if $a=0$. Then
\begin{equation}\label{eq:elem}
\mathcal R(a,b):=|a+b|^p-|a|^p-p|a|^{p-2}a\cdot b\ \ge\ \frac{p(p-1)}2\,|b|^2(|a|+|b|)^{p-2},
\end{equation}
where the right-hand side is interpreted as $0$ if $b=0$.
\end{lemma}

The preceding estimate will be useful in obtaining a remainder term for
the Hardy inequality. To apply this estimate to the gradient of $u$, we
use the ground-state transformation $u(x)=|x_k|^{-\frac{k-p}{p}}v(x)$. This transformation naturally leads to a quantity involving both
$\frac{|v(x)|}{|x_k|}$ and $|\nabla v(x)|$. We therefore introduce the
following mixed energy, which will be used to estimate the remainder
term.

\begin{definition}[Mixed energy]\label{def:J}
Let $N \geq 2, \, 1 \leq k<N$, $1<p<2$, and let
\begin{equation*}
    K=\left\{x=(x_k,x_{N -k})\in\mathbb{R}^k\times\mathbb{R}^{N-k} :x_k=0 \right\}.
\end{equation*}
For $v\in C_c^1(\mathbb{R}^N\setminus K)$, let $Z_v:=\{x\in\mathbb{R}^N:\nabla v(x)\neq0\}$. Then $Z_v$ is an open set. We define
\begin{equation}\label{Defn: M_v}
    M_v(x):= \left| \frac{k-p}{p} \right|
    \frac{|v(x)|}{|x_k|}
    +|\nabla v(x)|,
\end{equation}
and introduce the mixed energy
\begin{equation}\label{Defn: J(v)}
    J(v):=
    \int_{Z_v}
    |x_k|^{p-k}|\nabla v(x)|^2
    M_v(x)^{p-2}\,dx.
\end{equation}
\end{definition}

We now use the mixed energy defined above to obtain a remainder term in
the Hardy inequality. For this, we define the positive function $\omega(x) = |x_{k}|^{-\frac{k-p}{p}}$, which is a solution to the Euler--Lagrange equation associated with the
Hardy inequality \eqref{Hardy inequality on submanifold},
\begin{equation}\label{local Euler--Lagrange equation}
    - \operatorname{div} \big( |\nabla \omega(x)|^{p-2} \nabla \omega(x) \big)
    = \left| \frac{k-p}{p} \right|^{p} \frac{\omega(x)^{p-1}}{|x_{k}|^{p}}.
\end{equation}
We write $u=\omega v$ and use the quantitative estimate from
Lemma \ref{lem:elem} together with integration by parts. This gives the following well known Hardy inequality with a remainder.

\begin{lemma}[Hardy inequality with a remainder]\label{lem:gsr}
Let $N \geq 2, \, 1 \leq k<N$, and $1<p<2$. Then, for any $u\in C^{1}_{c}(\mathbb{R}^{N})$ when $p<k$, and $u\in C^{1}_{c}(\mathbb{R}^{N}\setminus K)$ when $p>k$, the following Hardy inequality with a remainder holds:
\begin{equation}
      \int_{\mathbb{R}^{N}} |\nabla u(x)|^{p} \,dx
-
\left| \frac{k-p}{p}\right|^{p}
\int_{\mathbb{R}^{N}}
\frac{|u(x)|^{p}}{|x_{k}|^p}\,dx 
    \geq
    \frac{p(p-1)}{2}J(v),
\end{equation}
where $J(v)$ is defined in \eqref{Defn: J(v)}, and  $v:=|x_k|^{ \frac{k-p}{p} } u$.
\end{lemma}

The remainder term from Lemma \ref{lem:gsr} is expressed in terms of
$J(v)$. In order to connect this quantity with the weighted Sobolev
norms that appear in Theorem \ref{Theorem: HSM Plane Singularity}, we introduce the following weighted integrals. The choice of the exponent $b(s)$ is made according to the scaling of the problem.

\begin{definition}\label{remark on T_q}
    For $s>0$ and $v \in C^{1}_{c}(\mathbb{R}^{N} \setminus K)$, we define
\begin{equation}\label{Defn: T_s}
    T_s(v):=
    \int_{\mathbb{R}^{N}} |x_k|^{b(s)}|v(x)|^s\, d x,
\end{equation}
where 
\begin{equation}\label{b(s)}
b(s):= \frac{(N-k)s}{p} - N.
\end{equation}
If we write $v(x)= |x_{k}|^{\frac{k-p}{p}} u(x)$, where $u \in C^{1}_{c}(\mathbb{R}^{N} \setminus K) $, then for any $q>1$,
\begin{equation*}
    T_{q}(v) = \int_{\mathbb{R}^{N}} \frac{|u(x)|^{q}}{|x_{k}|^{N+ (p-N)\frac{q}{p}}}\,dx.
\end{equation*}
\end{definition}

\section{A weighted \texorpdfstring{$L^1$}{L1} Sobolev inequality and some key estimates}\label{Section3}

In this section, we establish the weighted Sobolev estimates that will
be used to prove Theorem \ref{Theorem: HSM Plane Singularity}. We use a weighted $L^1$ Hardy inequality and combine it with the endpoint Sobolev inequality to obtain
a weighted Sobolev inequality. We then record some relations between the
parameters used in the proof and establish the key estimate that connects
the mixed energy $J(v)$ with the weighted $L^s$ quantities introduced
in the previous section. Finally, we prove a cut-off result for the case
$p<k$, which allows us to reduce the proof of Theorem \ref{Theorem: HSM Plane Singularity} to functions supported away from the singular set $K$.

\smallskip

We first recall a weighted $L^1$ Hardy inequality which will be used to obtain the weighted Sobolev inequality below. The weight is chosen according to the codimension of the singular set $K$. This result is well known; for its proof and generalizations see for example \cite{MR4952332}. 

\begin{lemma}[$L^1$ Hardy inequality in codimension $k$]\label{lem:L1H}
Let $A\in\mathbb{R}$ satisfy $ A+k-1>0$. Then, for every $w\in C^{1}_c(\mathbb{R}^{N} \setminus K)$,
\begin{equation}
    \int_{\mathbb{R}^{N}} |x_k|^{A-1}|w(x)|\,dx
    \leq
    \frac{1}{A+k-1}
    \int_{\mathbb{R}^{N}} |x_k|^A|\nabla w(x)|\,dx.
\end{equation}
\end{lemma}

We next recall the endpoint Gagliardo--Nirenberg--Sobolev inequality, which will be combined with the weighted $L^1$ Hardy inequality above to obtain the required weighted Sobolev inequality. For the proof, we refer to \cite[\S5.6.1, Theorem $1$]{Evans}, \cite{Gagliardo}, and \cite{Nirenberg}.

\begin{theorem}[Gagliardo--Nirenberg--Sobolev inequality, $p=1$]
\label{thm:GNS}
Let $N\geq2$. Then there exists a constant $C_N>0$, depending only on
$N$, such that
\begin{equation}\label{GNS}
    \|g\|_{L^{\frac{N}{N-1}}(\mathbb{R}^N)}
    \leq
    C_N\|\nabla g\|_{L^1(\mathbb{R}^N)}
\end{equation}
for every $g\in\Cc(\mathbb{R}^N)$. The sharp constant in \eqref{GNS} is $C_N=\frac{\Gamma\left(\frac{N}{2}+1\right)^{1/N}}{N\sqrt{\pi}}$.
\end{theorem}

We now combine the weighted $L^1$ Hardy inequality with the endpoint Sobolev inequality to obtain the following weighted Sobolev inequality.

\begin{lemma}[Weighted $L^1$ Sobolev inequality]\label{lem:L1sob}
Let $A+k-1>0$, $1\leq r\leq\frac{N}{N-1}$, and define $B:=(N+A-1)r-N$. Set
\begin{equation*}
    c_1:=\frac{1}{A+k-1},
\quad c_2:=C_N\left(1+\frac{|A|}{A+k-1}\right),
    \quad \text{and} \quad
    \theta:=(r-1)(N-1)\in[0,1].
\end{equation*}
Then, for every $w\in C^{1}_{c}(\mathbb{R}^{N} \setminus K)$,
\begin{equation}
    \left(
        \int_{\mathbb{R}^{N}}|x_k|^B|w(x)|^r\,dx
    \right)^{\frac{1}{r}}
    \leq
    C_1 \int_{\mathbb{R}^{N}} |x_k|^A|\nabla w(x)|\, d x, 
\end{equation}
where $C_1:= \left( c_1^{1-\theta} c_2^{\theta \frac{N}{N-1}} \right)^{\frac{1}{r}}$.
\end{lemma}

\begin{proof}
We first consider the two endpoint cases. If $r=1$, then $B=A-1$, and the desired estimate follows immediately from
Lemma \ref{lem:L1H},
\begin{equation*}
    \int_{\mathbb{R}^{N}} |x_k|^{A-1}|w(x)|\,dx
    \leq
    c_1 \int_{\mathbb{R}^{N}} |x_k|^A|\nabla w(x)|\, d x.
\end{equation*}

At the other endpoint, let $r=\frac{N}{N-1}$. In this case, $B=\frac{NA}{N-1}$. Define $g:=|x_k|^A w$. Since $w\in\ C^{1}_{c}(\mathbb{R}^{N} \setminus K)$, the function $g$ has compact support contained
in $\mathbb{R}^{N} \setminus K$ and may therefore be extended by zero to a function in
$\Cc(\mathbb{R}^N)$. On $\mathbb{R}^{N} \setminus K$, using $|\nabla |x_k||=1$, we have
\begin{align*}
    |\nabla g(x)|
    \leq
    |x_k|^A|\nabla w(x)|
    +
    |A||x_k|^{A-1}|w(x)|.
\end{align*}
Consequently, Lemma \ref{lem:L1H} gives
\begin{align*}
 \|\nabla g\|_{L^1(\mathbb{R}^N)} & \leq \int_{\mathbb{R}^{N}} |x_k|^A|\nabla w(x)|\, d x +  |A| \int_\Omega |x_k|^{A-1}|w(x)|\,dx \\ &   \leq \left(  1+\frac{|A|}{A+k-1} \right)   \int_{\mathbb{R}^{N}} |x_k|^A|\nabla w(x)|\, d x.
\end{align*}
Applying Theorem~\ref{thm:GNS}, we obtain
\begin{align*}
    \left(
        \int_\Omega
        |x_k|^{\frac{NA}{N-1}}
        |w(x)|^{\frac{N}{N-1}}
        \,dx
    \right)^{\frac{N-1}{N}}
   &  =
    \|g\|_{L^{\frac{N}{N-1}}(\mathbb{R}^N)} \\ &
    \leq
    C_N\|\nabla g\|_{L^1(\mathbb{R}^N)}
\leq
    c_2    \int_{\mathbb{R}^{N}} |x_k|^A|\nabla w(x)|\, d x.
\end{align*}

It remains to consider the intermediate range $1<r<\frac{N}{N-1}$. In this case $0<\theta<1$, and the definition of $\theta$ gives
\begin{equation*}
    (1-\theta)
    +\theta\frac{N}{N-1}
    =
    1+\frac{\theta}{N-1}
    =
    r.
\end{equation*}
Moreover,
\begin{align*}
    (1-\theta)(A-1)
    +\theta\frac{NA}{N-1}
    =
    A-1+\theta\frac{A+N-1}{N-1} =
    (N+A-1)r-N =B.
\end{align*}
It follows that
\begin{align*}
    |x_k|^B|w(x)|^r
    =
    \left(
        |x_k|^{A-1}|w(x)|
    \right)^{1-\theta}
    \left(
        |x_k|^{\frac{NA}{N-1}}
        |w(x)|^{\frac{N}{N-1}}
    \right)^\theta.
\end{align*}
Applying Hölder's inequality with conjugate exponents
$1/(1-\theta)$ and $1/\theta$, we obtain
\begin{align*}
    \int_{\R^N} |x_k|^B|w(x)|^r\,dx
    &\leq
    \left(
        \int_{\R^N} |x_k|^{A-1}|w(x)|\,dx
    \right)^{1-\theta}
    \left(
        \int_{\R^N}
        |x_k|^{\frac{NA}{N-1}}
        |w(x)|^{\frac{N}{N-1}}
        \,dx
    \right)^\theta \\
    &\leq
    \left(c_1\int_{\mathbb{R}^{N}} |x_k|^A|\nabla w(x)|\, d x\right)^{1-\theta}
    \left(c_2\int_{\mathbb{R}^{N}} |x_k|^A|\nabla w(x)|\, d x\right)^{\theta N/(N-1)} \\ 
    &=
    c_1^{1-\theta}
    c_2^{\theta \frac{N}{N-1}}
   \left( \int_{\mathbb{R}^{N}} |x_k|^A|\nabla w(x)|\, d x \right)^r.
\end{align*}
Taking the $r$-th root completes the proof of lemma.
\end{proof}

\begin{remark}
The relation $B=(N+A-1)r-N$ is precisely the relation dictated by the scaling of the weighted
inequality. The condition $A+k-1>0$ is the condition required for the
$L^1$ Hardy inequality in Lemma \ref{lem:L1H}. In particular, no
positivity assumption on $A$ itself is needed.
\end{remark}

We next record some relations between the parameters introduced above. These relations will be useful in the proof of the key estimate below. We omit its purely algebraic proof.

\begin{lemma}[Parameters]\label{lem:param}
Let $N \geq 2$, $1 \leq k<N$, $1<p<2$,
$p<q\leq p^{*}=\frac{Np}{N-p}$, and $p'= \frac{p}{p-1}$. Define
\begin{equation}\label{parameters}
s_a:=p+\frac{2(p-1)(q-p)}{p},
\quad
m:=\frac{s_a+p}{2},
\quad
A:=1-k+\frac{(N-k)(p-1)(q-p)}{p^2}.
\end{equation}
Then $p<s_a<q$ and $m=p+\frac{(p-1)(q-p)}{p}$. In particular, $m>p>1$ and $m>\frac{p}{2}$. Moreover, the following identities hold:
\begin{enumerate}[label=\textup{(\roman*)}]
\item $2m-p=s_a$;

\item $m-1=\frac{(p-1)q}{p}=\frac{q}{p'}$, so that $(m-1)p'=q$;

\item $2A+k-2=b(s_a)$;

\item $\frac{Ap-p+k}{p-1}=b(q)$;

\item $A+k-1=\frac{(N-k)(p-1)(q-p)}{p^2}>0$ and $N+A-1=\frac{(N-k)m}{p}$;

\item for $s\in\{s_a,q\}$, the number $r_s:=\frac{s}{m}$ satisfies $1<r_{s_a}<r_q\leq\frac{N}{N-1}$, and $(N+A-1)r_s-N=b(s)$,
\end{enumerate}
where $b(s)$ is defined in \eqref{b(s)}. 
\end{lemma}

The next estimate is the main analytic ingredient needed to control the
weighted first-order term by the mixed energy and the weighted $L^s$
quantities introduced above.

\begin{lemma}[Key estimate]\label{lem:key} Let $A$ and $m>1$ be as defined in \eqref{parameters}, and let  $v\in C^{1}_{c}(\mathbb{R}^{N} \setminus K)$. Define
\begin{equation*}
    I(v):=
    \int_{\mathbb{R}^{N}}
    |x_k|^A|v(x)|^{m-1}|\nabla v(x)|\, dx<\infty.
\end{equation*}
Then
\begin{equation}
    I(v)
    \leq
    \frac{2}{p}
    \left| \frac{k-p}{p} \right|^{\frac{2-p}{2}}
    J(v)^{\frac12}
    T_{s_a}(v)^{\frac12}
    +
    J(v)^{\frac1p}
    T_q(v)^{\frac1{p'}},
\end{equation}
where $s_a$ is defined in \eqref{parameters}, $p<q\leq p^*
=\frac{Np}{N-p}$, $J(v)$ is defined in \eqref{Defn: J(v)}, and
$T_s$ is defined in \eqref{Defn: T_s}.
\end{lemma}

\begin{proof}
Let $Z_{v} = \{ x  \in \mathbb{R}^{N} \setminus K : \nabla v(x) \neq 0 \} $. It suffices to consider the integral over $Z_{v}$.
On $Z_{v}$, we have $M_{v}(x)>0$, where $M_{v}$ is defined in \eqref{Defn: M_v}, and therefore
\begin{align*}
    |x_k|^A|v(x)|^{m-1}|\nabla v(x)|
    &=
    \left[
        |x_k|^{\frac{p-k}{2}}
        |\nabla v(x)|M_{v}(x)^{\frac{p-2}{2}}
    \right]
    \left[
        |x_k|^{A-\frac{p-k}{2}}
        |v(x)|^{m-1}M_{v}(x)^{\frac{2-p}{2}}
    \right],
\end{align*}
An application of the Cauchy--Schwarz inequality consequently gives
\begin{align*}
    I(v)
    \leq
    J(v)^{\frac{1}{2}}
    \left(
        \int_{Z_{v}}
        |x_k|^{2A-p+k}
        |v(x)|^{2m-2}
        M_{v}(x)^{2-p}\,dx \right)^{\frac{1}{2}},
\end{align*}
where $J(v)$ is defined in \eqref{Defn: J(v)}.

We now estimate the second factor. Using $ M_{v}(x)= \left| \frac{k-p}{p} \right| \frac{|v(x)|}{|x_k|} +|\nabla v(x)|$, which is defined in \eqref{Defn: M_v} and $0<2-p<1$, we have
\begin{equation*}
    M_{v}(x)^{2-p}
    \leq
    \left| \frac{k-p}{p} \right|^{2-p}
    |v(x)|^{2-p}|x_k|^{p-2}
    +
    |\nabla v(x)|^{2-p}.
\end{equation*}
It follows that
\begin{align*}
    \int_{Z_{v}}
    |x_k|^{2A-p+k}
    |v(x)|^{2m-2}
    M_{v}(x)^{2-p}\, dx
    \leq
    \left| \frac{k-p}{p} \right|^{2-p}K_1+K_2,
\end{align*}
where
\begin{align*}
    K_1
    &:=
    \int_{Z_{v}}
    |x_k|^{2A+k-2}
    |v(x)|^{2m-p}\,dx,\\
    K_2
    &:=
    \int_{Z_{v}}
    |x_k|^{2A-p+k}
    |v(x)|^{2m-2}
    |\nabla v(x)|^{2-p}\,dx.
\end{align*}
By Lemma \ref{lem:param} (i) and (iii), the first of these integrals is
precisely $K_1=T_{s_a}(v)$, where $T_{s_a}$ is defined in \eqref{Defn: T_s}.

For the second integral, we write its integrand as
\begin{align*}
    |x_k|^{2A-p+k}
    |v(x)|^{2m-2}
    |\nabla v(x)|^{2-p}
    =
    \left(
        |x_k|^A|v(x)|^{m-1}|\nabla v(x)|
    \right)^{2-p}
    |x_k|^{Ap-p+k}
    |v(x)|^{(m-1)p}.
\end{align*}
Hence, Hölder's inequality with conjugate exponents $\frac{1}{2-p}$ and  $\frac{1}{p-1}$ gives
\begin{align*}
    K_2
    \leq
    I(v)^{2-p}
    \left(
        \int_{\R^N}
        |x_k|^{\frac{Ap-p+k}{p-1}}
        |v(x)|^{(m-1)p'}\,dx
    \right)^{p-1}.
\end{align*}
By Lemma \ref{lem:param} (ii) and (iv), the last integral is
$T_q(v)$, and consequently
\begin{equation*}
    K_2
    \leq
    I(v)^{2-p}T_q(v)^{p-1}.
\end{equation*}
Combining the preceding estimates and using
$\sqrt{x+y}\leq\sqrt{x}+\sqrt{y}$ for $x,y\geq0$, we obtain
\begin{align*}
    I(v)
    &\leq
    \left| \frac{k-p}{p} \right|^{\frac{2-p}{2}}
    J(v)^{\frac12}
    T_{s_a}(v)^{\frac12}
    +
    J(v)^{\frac12}
    T_q(v)^{\frac{p-1}{2}}
    I(v)^{\frac{2-p}{2}}.
\end{align*}

It remains to absorb the last occurrence of $I$ into the left-hand side.
Young's inequality with conjugate exponents $P=\frac{2}{2-p}$ and $ P'=\frac{2}{p}$, applied with $x=I(v)^{\frac{2-p}{2}}$ and $y=J(v)^{\frac12}T_q(v)^{\frac{p-1}{2}}$, yields
\begin{align*}
J(v)^{\frac{1}{2}}T_q(v)^{\frac{p-1}{2}}
    I(v)^{\frac{2-p}{2}}
    \leq
    \frac{2-p}{2}I(v)
    +
    \frac{p}{2}
    J(v)^{\frac{1}{p}}
    T_q(v)^{\frac{1}{p'}}.
\end{align*}
Therefore,
\begin{align*}
    I(v)
    \leq
    \left| \frac{k-p}{p} \right|^{\frac{2-p}{2}}
    J(v)^{\frac{1}{2}}
    T_{s_a}(v)^{\frac12}
    +
    \frac{2-p}{2}I(v)
    +
    \frac{p}{2}
    J(v)^{\frac{1}{p}}
    T_q(v)^{\frac{1}{p'}}.
\end{align*}
Since $I(v)<\infty$, subtracting $\frac{2-p}{2}I$ from both sides and
multiplying by $\frac{2}{p}$ proves the lemma.
\end{proof}

The following lemma allows us to approximate functions in
$C_c^1(\mathbb{R}^N)$ by functions supported away from the singular
set $K$, in the case $p<k$. In particular, it shows that, in order to prove
Theorem \ref{Theorem: HSM Plane Singularity}, it is sufficient to consider
$u\in C_c^1(\mathbb{R}^N\setminus K)$.

\begin{lemma}[Cut-off near $K$]\label{lem:cutoff}
Let $p<k$ and $u\in\Cc(\R^N)$. Then there exists a sequence
$u_n\in C^{1}_{c}(\mathbb{R}^{N}\setminus K)$ such that
\begin{equation}
\int_{\mathbb{R}^N}|\nabla u_n(x)|^p\,dx
\longrightarrow
\int_{\mathbb{R}^N}|\nabla u(x)|^p\,dx,
\qquad
\int_{\mathbb{R}^N}\frac{|u_n(x)|^p}{|x_k|^p}\,dx
\longrightarrow
\int_{\mathbb{R}^N}\frac{|u(x)|^p}{|x_k|^p}\,dx,
\end{equation}
and $T_q(u_n)\longrightarrow  T_q(u)$, where $T_{q}$ is defined in \eqref{Defn: T_s}.
\end{lemma}

\begin{proof}
Fix $\vartheta\in C^\infty(\R)$ such that $0\leq\vartheta\leq1$,
$\vartheta=0$ on $(-\infty,1]$, $\vartheta=1$ on $[2,\infty)$, and
$|\vartheta'|\leq2$. Define
\begin{equation*}
\eta_n(x):=\vartheta(n|x_k|).
\end{equation*}
Then $\eta_n$ vanishes on the open set
$\{|x_k|<1/n\}$ containing $K$ and is smooth on
$\mathbb{R}^N\setminus K$. Setting $u_n:=\eta_nu$, we see that $\supp u_n\subset \supp u\cap\{|x_k|\geq1/n\}$, which is a compact subset of $\mathbb{R}^N\setminus K$. Hence
$u_n\in C_c^1(\mathbb{R}^N\setminus K)$. Moreover,
\begin{equation*}
|\nabla\eta_n(x)|
\leq
2n\,\mathbf{1}_{\{1/n\leq |x_k|\leq2/n\}}.
\end{equation*}

Choose $R>0$ such that $\supp u\subset Q_R:=B_R^k\times B_R^{N-k}$, and denote
$\omega_j:=|B_1^j|$. Since $|x_k|\leq 2/n$ on the support of
$\nabla\eta_n$, we obtain
\begin{align*}
\int_{\mathbb{R}^N}|u(x)\nabla\eta_n(x)|^p\,dx
&\leq
\|u\|_\infty^p(2n)^p
|B_{2/n}^k|\,|B_R^{N-k}| \\
&=
2^{p+k}\omega_k\omega_{N-k}
R^{N-k}\|u\|_\infty^p n^{p-k}
\longrightarrow0,
\end{align*}
because $p<k$.

On the other hand, $\eta_n\to1$ almost everywhere in $\mathbb{R}^N$.
Therefore, by dominated convergence,
\begin{equation*}
\|(\eta_n-1)\nabla u\|_{L^p(\mathbb{R}^N)}
\longrightarrow0.
\end{equation*}
Since
\begin{equation*}
\nabla u_n(x)-\nabla u(x)
=(\eta_n(x)-1)\nabla u(x)+u(x)\nabla\eta_n(x),
\end{equation*}
the preceding estimates and Minkowski's inequality imply that
$\nabla u_n\to\nabla u$ in $L^p(\mathbb{R}^N)$. Consequently,
\begin{equation*}
\|\nabla u_n\|_{L^p(\mathbb{R}^N)}
\longrightarrow
\|\nabla u\|_{L^p(\mathbb{R}^N)},
\end{equation*}
and hence
\begin{equation*}
\int_{\mathbb{R}^N}|\nabla u_n(x)|^p\,dx
\longrightarrow
\int_{\mathbb{R}^N}|\nabla u(x)|^p\,dx.
\end{equation*}
Finally, since $0\leq\eta_n(x)\leq1$ and $\eta_n\to1$ almost everywhere,
\begin{equation*}
\frac{|u_n(x)|^p}{|x_k|^p}
\leq
\frac{|u(x)|^p}{|x_k|^p},
\qquad
|u_n(x)|^q|x_k|^{-b(q)}
\leq
|u(x)|^q|x_k|^{-b(q)}.
\end{equation*} 
Thus, another application of the dominated
convergence theorem yields
\begin{equation*}
\int_{\mathbb{R}^N}\frac{|u_n(x)|^p}{|x_k|^p}\,dx
\longrightarrow
\int_{\mathbb{R}^N}\frac{|u(x)|^p}{|x_k|^p}\,dx
\end{equation*}
and $T_q(u_n)\longrightarrow T_q(u)$. This completes the proof.
\end{proof}

\section{Proof of the main theorem}\label{Section4}

In this section, we prove Theorem \ref{Theorem: HSM Plane Singularity} by combining the quantitative remainder estimate for the Hardy inequality established in Section 2 with the weighted Sobolev estimates obtained in Section 3.

\begin{proof}[\textbf{Proof of Theorem \ref{Theorem: HSM Plane Singularity}}] By Lemma \ref{lem:cutoff}, it is sufficient to prove the theorem for $u \in C^{1}_{c}(\mathbb{R}^{N} \setminus K)$. Let $u \in C^{1}_{c}(\mathbb{R}^{N} \setminus K)$, and we further assume that $u\not\equiv0$. Set
\begin{equation*}
    v:=|x_k|^{\frac{k-p}{p}} u.
\end{equation*}

We first apply the weighted Sobolev inequality to suitable powers of $v$. Let $m>1$, which is defined in \eqref{parameters}. Defining $ w:=|v|^m\in C_c^1(\mathbb{R}^N\setminus K)$ and using $\nabla w = m |v|^{m-2} v\nabla v$, we have
\begin{equation*}
\int_{\mathbb{R}^{N}} |x_{k}|^{A} |\nabla w(x)| \, dx = m\int_{\mathbb{R}^{N}} |x_{k}|^{A} |v(x)|^{m-1} |\nabla v(x)| \, dx = m I(v),
\end{equation*}
where $A$ is defined in \eqref{parameters}. For $s\in\{s_a,q\}$, Lemma \ref{lem:param} (v), (vi) allows us to apply
Lemma \ref{lem:L1sob} with $r=r_s = \frac{s}{m}$ and $B=b(s)$. Since
\begin{equation*}
|w|^{r_s}=|v|^{mr_s}=|v|^s,
\end{equation*}
we obtain
\begin{align}\label{ineq1}
    T_s(v)^{\frac{m}{s}} & = \left(
    \int_{\mathbb{R}^{N}} |x_k|^{b(s)}|v(x)|^s\,dx \right)^{\frac{m}{s}} = \left( \int_{\mathbb{R}^{N}} |x_k|^{b(s)}|w(x)|^{r_s}\,dx \right)^{\frac{m}{s}} \nonumber \\ & \leq \overline{C}_{1} \int_{\mathbb{R}^{N}} |x_{k}|^{A} |\nabla w(x)| \, dx  = \overline{C}_{1} m I(v),
\end{align}
where $\overline C_1$ denotes the larger of the two constants in
Lemma \ref{lem:L1sob} corresponding to $r=r_{s_a}$ and $r=r_q$.

For convenience, set
\begin{equation*}
    X_s:=T_s(v)^{\frac{1}{s}},
    \qquad
    \mathcal Z:=\max\{X_{s_a},X_q\}.
\end{equation*}
Since $v\not\equiv0$, we have $\mathcal{Z}>0$. Moreover, by
Lemma \ref{lem:param} (i), (ii), we have
\begin{equation*}
    \frac{s_a}{2}=m-\frac{p}{2},
    \qquad
    \frac{q}{p'}=m-1.
\end{equation*}
Both exponents are positive. Consequently, we have
\begin{align*}
    T_{s_a}(v)^{\frac{1}{2}}
    &=
    X_{s_a}^{\frac{s_a}{2}}
    =
    X_{s_a}^{m-\frac{p}{2}}
    \leq
    \mathcal Z^{m-\frac{p}{2}},\\
    T_q(v)^{\frac{1}{p'}}
    &=
    X_q^{\frac{q}{p'}}
    =
    X_q^{m-1}
    \leq
    \mathcal Z^{m-1}.
\end{align*}
We now combine these estimates with the key estimate of
Lemma \ref{lem:key}. Define
\begin{equation*}
    L:=
    \max\left\{
        \frac{2}{p}\left| \frac{k-p}{p} \right|^{\frac{2-p}{2}},1
    \right\},
    \qquad
    \Lambda:=\overline C_1mL.
\end{equation*}
Since $X_s^m=T_s(v)^{\frac{m}{s}}$, by Lemma \ref{lem:key}, the inequality \eqref{ineq1}, for each $s\in\{s_a,q\}$,
\begin{align*}
    X_s^m
    &\leq
    \Lambda
    \left(
        J(v)^{\frac{1}{2}}\mathcal Z^{m-\frac{p}{2}}
        +
        J(v)^{\frac{1}{p}}\mathcal Z^{m-1}
    \right).
\end{align*}
Taking the maximum over $s\in\{s_a,q\}$ and using
$\mathcal Z^m=\max\{X_{s_a}^m,X_q^m\}$, we obtain
\begin{equation*}
    \mathcal Z^m
    \leq
    \Lambda
    \left(
        J(v)^{\frac{1}{2}}\mathcal Z^{m-\frac{p}{2}}
        +
        J(v)^{\frac{1}{p}}\mathcal Z^{m-1}
    \right).
\end{equation*}
Since $\mathcal Z>0$, division by $\mathcal Z^m$ yields
\begin{equation*}
    1
    \leq
    \Lambda J(v)^{\frac{1}{2}}\mathcal Z^{-\frac{p}{2}}
    +
    \Lambda J(v)^{\frac{1}{p}}\mathcal Z^{-1}.
\end{equation*}

At least one of the two terms on the right-hand side is therefore at least
$1/2$. If the first term is at least $1/2$, then
\begin{equation*}
    \mathcal Z^{p}
    \leq
    (2\Lambda)^2J(v).
\end{equation*}
On the other hand, if the second term is at least $1/2$, then
\begin{equation*}
    \mathcal Z^{p}
    \leq
    (2\Lambda)^pJ(v).
\end{equation*}
Thus, in either case,
\begin{equation*}
    \mathcal Z^p
    \leq
    C_2J(v),
    \qquad
    C_2:=
    \max\left\{
        (2\Lambda)^2,
        (2\Lambda)^p
    \right\}.
\end{equation*}
Therefore, we have
\begin{equation*}
    T_q(u)^{\frac{p}{q}}=
    X_q^p
    \leq
    \mathcal Z^p
    \leq
    C_2J(v).
\end{equation*}
Finally, by Lemma \ref{lem:gsr} and the definition of $T_{q}$ (see Remark \ref{remark on T_q}),
\begin{align*}
\int_{\mathbb{R}^{N}} |\nabla u(x)|^{p} \,dx
-
\left| \frac{k-p}{p}\right|^{p}
\int_{\mathbb{R}^{N}}
\frac{|u(x)|^{p}}{|x_{k}|^p}\,dx \geq  C  \left(\int_{\mathbb{R}^{N}} \frac{|u(x)|^{q}}{|x_{k}|^{N+ (p-N)\frac{q}{p}}} \, dx \right)^{\frac{p}{q}}.
\end{align*}
where $C:=\frac{p(p-1)}{2C_2}$. This completes the proof of theorem.
\end{proof}

\section{Hardy--Sobolev--Maz'ya inequalities on arbitrary domains}\label{Section5}
In this section, we prove Theorem \ref{HSM domains p<2}, which establishes the HSM inequality on arbitrary domains. Our approach follows the one-dimensional method of Frank and Loss \cite{FrankLoss2012}. We first establish a one-dimensional Hardy inequality with a remainder term and then prove an auxiliary estimate that controls the supremum norm of a function in terms of the Hardy deficit and its $L^q$ norm. Combining these estimates with a suitable slicing argument, we obtain the desired HSM inequality for arbitrary domains.

\smallskip

Let $f\in C^{1}_{c}\left((-1,1)\right)$, and define halves $H_{+}:=[0,1)$ and $H_{-}:=(-1,0]$. Observe that
\begin{equation*}
    g(s):=f(s) (\omega(s))^{-1}, 
\end{equation*}
where $\omega(s):=(1-|s|)^{\frac{p-1}{p}}$,
is continuous on $(-1,1)$, it is $C^{1}$ on $H_{+}$ and on $H_{-}$ (with one-sided derivatives at $0$), and it has compact support in $(-1,1)$. On each half, we set
\begin{equation}\label{M_g defn one dimension}
    M_{g}(s):=\left( \frac{p-1}{p} \right) \frac{|g(s)|}{1-|s|}+|g'(s)|, \,
    J_{\pm}(g):=\int_{H_{\pm}\cap\{g'\neq0\}}(1-|s|)^{p-1}|g'(s)|^{2}M_{g}(s)^{p-2}\,ds,
\end{equation}
and 
\begin{equation}\label{Defn: J_kappa}
    J(g):=J_{+}(g)+J_{-}(g),
\end{equation}
which is the one-dimensional analogue of the mixed energy \eqref{Defn: J(v)} for $k=1$

The next lemma, establishes Hardy inequality with  remainder terms on $(-1,1)$. 
Using Lemma \ref{lem:elem} with $a:= g \omega'$ and $b:= \omega  g' $, so that $a+b = f'= (g \omega)'$ on each closed halves $H_{\pm}$, we have the following lemma. Its proof is elementary and uses integration by parts.

\begin{lemma}[Hardy inequality with remainder terms on $(-1,1)$]\label{lem:gs1d} Let $1<p<2$, and let $f\in C^{1}_{c}\left((-1,1)\right)$. Then,
\begin{equation*}
\int_{-1}^{1}|f'(s)|^{p}\,ds- \left( \frac{p-1}{p} \right)^{p}\int_{-1}^{1}\frac{|f(s)|^{p}}{(1-|s|)^{p}}\,ds
    \geq 2 \left( \frac{p-1}{p} \right)^{p-1}|g(0)|^{p}+\frac{p(p-1)}{2}J_{}(g),
\end{equation*}
where $g(s)= f(s)\omega^{-1}(s) =  f(s)(1-|s|)^{-\frac{p-1}{p}}$ and $J$ is defined in \eqref{Defn: J_kappa}.
\end{lemma}
The following one-dimensional result is the key to proof Theorem \ref{HSM domains p<2}. Once it is established, the proof of Theorem \ref{HSM domains p<2} will be analogous to the one in \cite{FrankLoss2012}.
\begin{proposition}[Analogue of {\cite[Proposition 2.5]{FrankLoss2012}} for $1<p<2$]\label{prop:1d}
Let $1<p<2$ and $q>0$. There is a constant $C_{p,q}>0$ such that for every $f\in\Cc\left((-1,1)\right)$,
\begin{align*}
\sup_{t\in[-1,1]} &|f(t)|^{\,q(p-1)+p}\ \\ & \le\ C_{p,q}\ \left( \int_{-1}^{1}|f'(s)|^{p}\,ds- \left( \frac{p-1}{p} \right)^{p}\int_{-1}^{1}\frac{|f(s)|^{p}}{(1-|s|)^{p}}\,ds \right) \Big(\int_{-1}^1|f(s)|^q \, ds\Big)^{p-1}.
\end{align*}
\end{proposition}
\begin{proof}
    Let $g=f \omega^{-1}$, where $\omega(s)=(1-|s|)^{\frac{p-1}{p}}$. Define $r(s):=1-|s|$, for all $s \in (-1,1)$. Throughout the proof, we assume
\begin{equation*}
    \al:=\frac{p-1}{p},\qquad\mu:=\frac{q}{p'}=\left( \frac{p-1}{p} \right) q, \qquad \beta:=\left( \frac{p-1}{p} \right)(\mu+1),
\end{equation*}
so that $p(\mu+1)=q(p-1)+p$, and $C$ denotes a constant depending only on $p$ and $q$ whose value may change from line to line. Since $f=\omega g$, we have
\begin{equation}\label{eq:fg}
|f(s)|^{q}=|\omega (s)g(s)|^{q} =  r(s)^{\mu} |g(s)|^{q},\quad |f(s)|^{\mu+1}= |\omega(s)g(s)|^{\mu+1} =  r(s)^{\beta}|g(s)|^{\mu+1}.
\end{equation}
We treat the half $[0,1)$ in detail; the half $(-1,0]$ is handled identically by symmetry.

First, we provide a pointwise bound for $f$. Let $t\in[0,1)$. Using $(|g|^{\gamma})'= \gamma |g|^{\gamma-2}gg'$ with $\gamma=\mu+1>1$, we have
\begin{equation*}
|g(t)|^{\mu+1}\leq|g(0)|^{\mu+1}+\int_{0}^{t}\left|\big(|g(s)|^{\mu+1}\big)'\right|ds=|g(0)|^{\mu+1}+(\mu+1)\int_{0}^{t}|g(s)|^{\mu}|g'(s)|\,ds.
\end{equation*}
We multiply both sides by $(1-t)^{\beta}$. Since $(1-t)^{\beta}\leq1$, and $(1-t)^{\beta}\leq(1-s)^{\beta} =r(s)^{\beta}$ for $0\leq s\leq t$, we obtain from \eqref{eq:fg}
\begin{equation}\label{eq:ptw}
|f(t)|^{\mu+1}\leq|f(0)|^{\mu+1}+(\mu+1)\int_{0}^{1}r(s)^{\beta}|g(s)|^{\mu}|g'(s)|\,ds.
\end{equation}
Assume $g'(s) \neq 0$, we have $M_{g}(s)>0$, where $M_g$ is defined in \eqref{M_g defn one dimension} and
\begin{equation*}
r(s)^{\beta}|g(s)|^{\mu}|g'(s)|=\left[r(s)^{\frac{p-1}{2}}|g'(s)|M_{g}(s)^{\frac{p-2}{2}}\right]\cdot\left[r(s)^{\beta-\frac{p-1}{2}}|g(s)|^{\mu}M_{g}(s)^{\frac{2-p}{2}}\right].
\end{equation*}
The Cauchy--Schwarz inequality gives
\begin{equation*}
    \left(\int_{0}^{1}r(s)^{\beta}|g(s)|^{\mu}|g'(s)|\,ds\right)^{2}\leq J_{+}(g)\int_{0}^{1}r(s)^{2\beta-p+1}|g(s)|^{2\mu}M_{g}(s)^{2-p}\,ds,
\end{equation*}
where $J_{+}$ is defined in \eqref{M_g defn one dimension}. Using $2-p\in(0,1)$, we have
\begin{align*}
    M_{g}(s)^{2-p} & = \left( \frac{p-1}{p}\frac{|g(s)|}{r(s)}+|g'(s)| \right)^{2-p} \\ & \leq \left( \frac{p-1}{p} \right)^{2-p}|g(s)|^{2-p}r(s)^{p-2}+|g'(s)|^{2-p}.
\end{align*}
Hence
\begin{align}\label{eq:split}
\left(\int_{0}^{1}r(s)^{\beta}|g(s)|^{\mu}|g'(s)|\,ds\right)^{2} & \leq J_{+}(g)\Bigg(\left( \frac{p-1}{p} \right)^{2-p}\int_{0}^{1}r(s)^{2\beta-1}|g(s)|^{2\mu+2-p}\,ds \nonumber \\ & \quad +\int_{0}^{1}r(s)^{2\beta-p+1}|g(s)|^{2\mu}|g'(s)|^{2-p}\,ds\Bigg).
\end{align}
We will now deal with the second integral above. Its integrand factors as
\begin{equation*}
    r(s)^{2\beta-p+1}|g(s)|^{2\mu}|g'(s)|^{2-p}=\left(r(s)^{\beta}|g(s)|^{\mu}|g'(s)|\right)^{2-p}\cdot r(s)^{\beta p-p+1}|g(s)|^{\mu p},
\end{equation*}
Applying H\"older's inequality with the conjugate exponents $\frac{1}{2-p}$ and $\frac{1}{p-1}$ and using \eqref{eq:fg}, gives
\begin{align}\label{(c)}
\nonumber\int_{0}^{1}r(s)^{2\beta-p+1}|g(s)|^{2\mu}|g'(s)|^{2-p}\,ds & \leq\left(\int_{0}^{1}r(s)^{\beta}|g(s)|^{\mu}|g'(s)|\,ds\right)^{2-p} \\ & \qquad \nonumber\times \left(\int_{0}^{1}r(s)^{\mu}|g(s)|^{q}\,ds\right)^{p-1}\\ & \leq\left(\int_{0}^{1}r(s)^{\beta}|g(s)|^{\mu}|g'(s)|\,ds\right)^{2-p}\left(\int_{-1}^{1}|f(s)|^{q}\,ds\right)^{p-1}.
\end{align}
For the first integral in \eqref{eq:split}, using the definition of $\mu$ and $\beta$, we have
\begin{equation*}
r(s)^{2\beta-1}|g(s)|^{2\mu+2-p}=\left(r(s)^{\mu}|g(s)|^{q}\right)^{p-1}\left(r(s)^{\beta-1}|g(s)|^{\mu+1}\right)^{2-p}.
\end{equation*}
H\"older's inequality with the exponents $\frac{1}{p-1}$ and $\frac{1}{2-p}$, together with \eqref{eq:fg}, gives
\begin{equation}\label{(d)}
\int_{0}^{1}r(s)^{2\beta-1}|g(s)|^{2\mu+2-p}\,ds \leq\left(\int_{-1}^{1}|f(s)|^{q}\,ds\right)^{p-1}\left(\int_{0}^{1}r(s)^{\beta-1}|g(s)|^{\mu+1}\,ds\right)^{2-p}.
\end{equation}

For the integral above, we apply an $L^{1}$ Hardy inequality with a boundary term. Since $\frac{d}{ds}\big[-\frac{(1-s)^{\beta}}{\beta}\big]=(1-s)^{\beta-1} = r(s)^{\beta-1}$ and $|g|^{\mu+1}$ vanishes near $1$, integration by parts and using $(|g|^{\mu+1})' = (\mu +1) |g|^{\mu-1} gg'$ gives
\begin{align}\label{eq:L1hardy}
\int_{0}^{1}r(s)^{\beta-1}|g(s)|^{\mu+1}\,ds & =\frac{|g(0)|^{\mu+1}}{\beta}+\frac{1}{\beta}\int_{0}^{1}r(s)^{\beta}\big(|g(s)|^{\mu+1}\big)'\,ds \nonumber \\ &    \leq\frac{1}{\beta}\left(|f(0)|^{\mu+1}+(\mu+1)\int_{0}^{1}r(s)^{\beta}|g(s)|^{\mu}|g'(s)|\,ds\right).
\end{align}

Combining \eqref{(d)} with \eqref{eq:L1hardy} and $\frac{\alpha}{\beta}=\frac{1}{\mu+1}$, we have
\begin{align*}
    \left( \frac{p-1}{p} \right)^{2-p}\int_{0}^{1} & r(s)^{2\beta-1}|g(s)|^{2\mu+2-p}\,ds \\ & \leq\left(\int_{-1}^{1}|f(s)|^{q}\,ds\right)^{p-1}\left(\frac{|f(0)|^{\mu+1}}{\mu+1}+\int_{0}^{1}r(s)^{\beta}|g(s)|^{\mu}|g'(s)|\,ds\right)^{2-p}.
\end{align*}
The quantity in the last bracket is at most $|f(0)|^{\mu+1}+\int_{0}^{1}r(s)^{\beta}|g(s)|^{\mu}|g'(s)|\,ds$, and the same quantity bounds $\int_{0}^{1}r(s)^{\beta}|g(s)|^{\mu}|g'(s)|\,ds$ in \eqref{(c)}. Inserting this and \eqref{(c)} into \eqref{eq:split}, and using $J_{+}(g)\leq J(g)$, we get
\begin{align}\label{eq:collect} \left(\int_{0}^{1}r(s)^{\beta}|g(s)|^{\mu}|g'(s)|\,ds\right)^{2} & \leq 2\,J(g)\left(\int_{-1}^{1}|f(s)|^{q}\,ds\right)^{p-1} \nonumber \\ & \quad \times\left(|f(0)|^{\mu+1}+\int_{0}^{1}r(s)^{\beta}|g(s)|^{\mu}|g'(s)|\,ds\right)^{2-p}.
\end{align}

Therefore, the inequality \eqref{eq:collect} gives
\begin{align*}
|f(0)|^{\mu+1}+ & \int_{0}^{1}r(s)^{\beta}|g(s)|^{\mu}|g'(s)|\,ds\leq|f(0)|^{\mu+1}\\ &
    +\sqrt{2}\,J(g)^{\frac{1}{2}}\left(\int_{-1}^{1}|f(s)|^{q}\,ds\right)^{\frac{p-1}{2}}\left(|f(0)|^{\mu+1}+\int_{0}^{1}r(s)^{\beta}|g(s)|^{\mu}|g'(s)|\,ds\right)^{\frac{2-p}{2}}.
\end{align*}
We will now apply the Young's inequality to the last two factors, in the following form: for $x,y\geq0$ and $\theta\in(0,1)$, $ y\,x^{\theta}\leq \theta x+(1-\theta)\,y^{\frac{1}{1-\theta}}$. Take $\theta=\frac{2-p}{2}$ , so that $y^{\frac{1}{1-\theta}}=y^{\frac2p}=2^{\frac{1}{p}}J(g)^{\frac{1}{p}}\left(\int_{-1}^{1}|f(s)|^{q}\,ds\right)^{\frac{1}{p'}}$. Subtracting $\frac{2-p}{2}$ times the finite left-hand side and multiplying by $\frac2p$, we obtain
\begin{align}\label{eq:X}
|f(0)|^{\mu+1}+\int_{0}^{1}r(s)^{\beta}|g(s)|^{\mu}|g'(s)|\,ds    \leq\frac{2}{p}|f(0)|^{\mu+1}+2^{\frac{1}{p}}J(g)^{\frac{1}{p}}\left(\int_{-1}^{1}|f(s)|^{q}\,ds\right)^{\frac{1}{p'}}.
\end{align}
The same bound holds on $(-1,0]$.

\medskip

Finally, by \eqref{eq:ptw}, \eqref{eq:X} and symmetry,
\begin{equation}\label{eq:sup}
\sup_{t \in [-1,1]}|f(t)|^{\mu+1}\leq\left(1+\frac{2(\mu+1)}{p}\right)|f(0)|^{\mu+1}
    +(\mu+1)2^{\frac{1}{p}}J(g)^{\frac{1}{p}}\left(\int_{-1}^{1}|f(s)|^{q}\,ds\right)^{\frac{1}{p'}}.
\end{equation}

\medskip

Put $\varepsilon:=2^{-\left(1+q+ \frac{p-1}{p} q \right)}$. We claim that there exists $C>0$ such that
\begin{equation}\label{eq:mid}
  |f(0)|^{\mu+1}\leq C\left(|g(0)|+J(g)^{\frac{1}{p}}\right)\left(\int_{-1}^{1}|f(s)|^{q}\,ds\right)^{\frac{1}{p'}}.
\end{equation}
The justification of this inequality will be divided into two cases.

\medskip

\emph{Case A: $\int_{-1}^{1}|f(s)|^{q}\,ds\geq\varepsilon|g(0)|^{q}$} (this includes $g(0)=0$). Since $\mu=\frac{q}{p'}$ and $f(0)=g(0)$,
\begin{equation*}
    |f(0)|^{\mu+1}=|g(0)|\cdot|g(0)|^{\frac{q}{p'}}\leq\varepsilon^{-\frac{1}{p'}}|g(0)|\left(\int_{-1}^{1}|f(s)|^{q}\,ds\right)^{\frac{1}{p'}},
\end{equation*}
which gives \eqref{eq:mid}.

\emph{Case B: $\int_{-1}^{1}|f(s)|^{q}\,ds<\varepsilon|g(0)|^{q}$.} Then $g(0)\neq0$. For $s\in[0,\frac12]$ we have $r(s)\geq\frac12$, hence $r(s)^{-\frac{p-1}{p} q}\leq2^{\frac{p-1}{p} q}$. By \eqref{eq:fg},
\begin{equation*}
\int_{0}^{1/2}|g(s)|^{q}\,ds\leq2^{\frac{p-1}{p} q}\int_{-1}^{1}|f(s)|^{q}\,ds<2^{\frac{p-1}{p}q}\varepsilon|g(0)|^{q}=2^{-q-1}|g(0)|^{q}.
\end{equation*}
If $|g(s)|>\frac{|g(0)|}{2}$ held on all of $[0,\frac12]$, the left-hand side would be larger than $\frac12\cdot2^{-q}|g(0)|^{q}$, a contradiction. So there is $s_{0}\in[0,\frac12]$ with $|g(s_{0})|\leq\frac{|g(0)|}{2}$. Therefore,
\begin{align} \label{B1}
\nonumber\frac{1}{2} |f(0)|^{\mu+1}\leq\big(1-2^{-(\mu+1)}\big)|g(0)|^{\mu+1} &  \leq|g(0)|^{\mu+1}-|g(s_{0})|^{\mu+1} \\ & \leq(\mu+1)\int_{0}^{1/2}|g(s)|^{\mu}|g'(s)|\,ds.
\end{align}

We repeat previous steps on $[0,\frac{1}{2}]$ without weights. For $r(s)\in[\frac12,1]$, we have $r(s)^{1-p}\leq2^{p-1}$ and $\left( \frac{p-1}{p} \right) \frac{|g(s)|}{r}\leq2\left( \frac{p-1}{p} \right)|g(s)|\leq|g(s)|$, because $2\left( \frac{p-1}{p} \right)<1$. Hence $M_{g}(s)^{2-p}\leq|g(s)|^{2-p}+|g'(s)|^{2-p}$, where $M_{g}$ is defined in \eqref{M_g defn one dimension} and $0<2-p<1$, and the Cauchy--Schwarz inequality, gives
\begin{align*}
\left(\int_{0}^{1/2}|g(s)|^{\mu}|g'(s)|\,ds\right)^{2} &  \\ & \hspace{-10mm} \leq2^{p-1}J(g)\left(\int_{0}^{1/2}|g(s)|^{2\mu+2-p}\,ds+\int_{0}^{1/2}|g(s)|^{2\mu}|g'(s)|^{2-p}\,ds\right).
\end{align*}
The same H\"older arguments as before, with $r(s)$ replaced by $1$, and the bound $\int_{0}^{1/2}|g(s)|^{q}\,ds\leq2^{\frac{p-1}{p} q}\int_{-1}^{1}|f(s)|^{q}\,ds$ give
\begin{align}\label{eq:loc}
\left(\int_{0}^{1/2}|g(s)|^{\mu}|g'(s)|\,ds\right)^{2} & \leq 2^{p-1}2^{\frac{p-1}{p} q(p-1)}J(g)\left(\int_{-1}^{1}|f(s)|^{q}\,ds\right)^{p-1} \nonumber \\ &\quad
\times\left[\left(\int_{0}^{1/2}|g(s)|^{\mu+1}\,ds\right)^{2-p}+\left(\int_{0}^{1/2}|g(s)|^{\mu}|g'(s)|\,ds\right)^{2-p}\right].
\end{align}
Denote 
\begin{equation*}
V:=|f(0)|^{\mu+1}+J(g)^{\frac1p}\left(\int_{-1}^{1}|f(s)|^{q}\,ds\right)^{\frac1{p'}}. 
\end{equation*} 
First, since $\frac2p<2$ and $2^{\frac1p}<2$, the estimate \eqref{eq:X} gives
\begin{equation}\label{B3}
\int_{0}^{1}r(s)^{\beta}|g(s)|^{\mu}|g'(s)|\,ds \leq 4V.
\end{equation}
Second, on $[0,\frac12]$ we have $r(s)^{\beta}\geq2^{-\beta}$, hence $\int_{0}^{1/2}|g(s)|^{\mu}|g'(s)|\,ds\leq2^{\beta+2} \, V$. Third, on $[0,\frac12]$ we have $r(s)^{\beta-1}\geq\min\{1,2^{1-\beta}\}$, so by \eqref{eq:L1hardy},
\begin{equation*}  \int_{0}^{1/2}|g(s)|^{\mu+1}\,ds\leq\frac{1+4(\mu+1)}{\beta\min\{1,2^{1-\beta}\}}V.
\end{equation*}
Inserting these three bounds in \eqref{eq:loc}, and using \eqref{B1}, we find
\begin{equation}\label{aa}
    |f(0)|^{2(\mu+1)}\leq C\,J(g)\left(\int_{-1}^{1}|f(s)|^{q}\,ds\right)^{p-1} \, V^{2-p}.
\end{equation}

To summarize: if $|f(0)|^{\mu+1}\leq J(g)^{\frac1p}\left(\int_{-1}^{1}|f(s)|^{q}\,ds\right)^{\frac{1}{p'}}$, then \eqref{eq:mid} holds with $C=1$. Otherwise $V<2|f(0)|^{\mu+1}$, so $V^{2-p}\leq2^{2-p}|f(0)|^{(\mu+1)(2-p)}$, and \eqref{aa} gives
\begin{equation*}
    |f(0)|^{(\mu+1)p}\leq C\,J(g)\left(\int_{-1}^{1}|f(s)|^{q}\,ds\right)^{p-1} =C\left(J(g)^{\frac1p}\left(\int_{-1}^{1}|f(s)|^{q}\,ds\right)^{\frac1{p'}}\right)^{p},
\end{equation*}
because $\frac{p}{p'}=p-1$. Taking the $p$-th root gives \eqref{eq:mid}.

\medskip

\noindent Finally, we insert \eqref{eq:mid} into \eqref{eq:sup}. This gives
\begin{equation*}
\sup_{ t \in [-1,1]}|f(t)|^{\mu+1} \leq C\left(|g(0)|+J(g)^{\frac1p}\right)\left(\int_{-1}^{1}|f(s)|^{q}\,ds\right)^{\frac{1}{p'}}.
\end{equation*}
By the convexity of $t\mapsto t^{p}$ we have $x+y\leq2^{\frac1{p'}}(x^{p}+y^{p})^{\frac1p}$ for $x,y\geq0$. Raising to the power $p$, and recalling $p(\mu+1)=q(p-1)+p$ and $\frac{p}{p'}=p-1$, we find
\begin{equation*}
    \sup_{t \in [-1,1]}|f(t)|^{\,q(p-1)+p}\leq C\left(|g(0)|^{p}+J(g)\right)\left(\int_{-1}^{1}|f(s)|^{q}\,ds\right)^{p-1}.
\end{equation*}
Finally, Lemma \ref{lem:gs1d} gives 
\begin{equation*}
\int_{-1}^{1}|f'(s)|^{p}\,ds- \left( \frac{p-1}{p} \right)^{p}\int_{-1}^{1}\frac{|f(s)|^{p}}{(1-|s|)^{p}}\,ds   \geq\min\left\{2\alpha^{p-1},\frac{p(p-1)}{2}\right\}\left(|g(0)|^{p}+J(g)\right),
\end{equation*}
which proves the proposition.
\end{proof}
\begin{corollary}\label{Cor53}
Let $1<p<2$ and $q>0$. Let $\Omega\subset\R$ be an open, proper set. Then, for all $u\in C_c^1(\Om)$,
\begin{align*}
\sup_{t\in\Om}&|f(t)|^{q(p-1)+p}  \ \\ & \le\ C_{p,q}\ \left( \int_{\Om}|f'(s)|^{p}\,ds- \left( \frac{p-1}{p} \right)^{p}\int_{\Om}\frac{|f(s)|^{p}}{\dist(s,\Om^c)^p}\,ds \right) \left(\int_{\Om}|f(s)|^q \, ds\right)^{p-1},  
\end{align*}
where $C_{p,q}>0$ is the same as in Proposition \ref{prop:1d}.
\end{corollary}
\begin{proof}
 If $\Om$ is an interval, the result follows from Proposition \ref{prop:1d} by scaling and translation. For general open set, we write is as a countable union of disjoint intervals and the extension is straightforward.
\end{proof}
\begin{proof}[Proof of Theorem \ref{HSM domains p<2}]
As we already mentioned, the proof of Theorem \ref{HSM domains p<2} follows \cite[Theorem 1.3]{FrankLoss2012}. Nevertheless, we provide its sketch for clarity of the present work. Let $e_1,\dots,e_N$ be a standard basis in $\R^N$, let $d_j:=d_{e_j}$, where $d_{e_j}$ is defined by \eqref{def: de}. Further, we denote $\widetilde{x}_j=(x_1,\dots,x_{j-1},x_{j+1},\dots,x_N)\in\R^{N-1}$, and, for $u\in C_c^1(\Om)$,
\begin{align*}
g_j(\widetilde{x}_j):=\int_{\R}\left(\left|\frac{\partial u}{\partial x_j}(x)\right|^p-\left(\frac{p-1}{p}\right)^p\frac{|u(x)|^p}{d_j(x)^p}\right)\,dx_j,\quad h_j(\widetilde{x}_j)=\int_{\R}|u(x)|^q\,dx_j.    
\end{align*}
Applying Corollary \ref{Cor53}, for every $x\in \Om$ and $1\le j\le N$ we have 
$$
|u(x)|\le \left(C_{p,q} g_{j}(\widetilde{x}_j) h_j(\widetilde{x}_j)^{p-1}\right)^{\frac{1}{q(p-1)+p}}.
$$
Thus, for some constant $C>0$,
$$
|u(x)|^q\le C \prod_{j=1}^N\left(g_{j}(\widetilde{x}_j) h_j(\widetilde{x}_j)^{p-1}\right)^{\frac{1}{p(N-1)}}.
$$
Integrating the above inequality over $\R^N$ and then applying Loomis–Whitney inequality \cite{MR31538} together with Hölder's and arithmetic-geometric mean inequality, we estimate (up to a multiplicative constant)
\begin{align*}
\|u\|_{L^{q}(\R^N)}^q&\lesssim \int_{\R^N}\prod_{j=1}^N\left(g_{j}(\widetilde{x}_j) h_j(\widetilde{x}_j)^{p-1}\right)^{\frac{1}{p(N-1)}}\,dx\le \prod_{j=1}^N\left(\int_{\R^{N-1}}g_j(\widetilde{x}_j)^{\frac{1}{p}}h_j(\widetilde{x}_j)^{\frac{1}{p'}}\,d\widetilde{x}_j\right)^{\frac{1}{N-1}}\\
&\le \prod_{j=1}^N\left(\|g_j\|_{L^1(\R^{N-1})}^\frac{1}{p}\|h_j\|_{L^1(\R^{N-1})}^\frac{1}{p'}\right)^{\frac{1}{N-1}}\\
&= \|u\|_{L^q(\R^N)}^{\frac{Nq}{(N-1)p'}}\prod_{j=1}^N\|g_j\|_{L^1(\R^{N-1})}^\frac{1}{(N-1)p}\le \|u\|_{L^q(\R^N)}^{\frac{Nq}{(N-1)p'}}\left(\frac{1}{N}\sum_{j=1}^N\|g_j\|_{L^1(\R^{N-1})}\right)^{\frac{N}{(N-1)p}}.
\end{align*}
Using $q\left(1-\frac{N}{(N-1)p'}\right)=\frac{N}{N-1}$, rearranging gives
\begin{equation}\label{ll}
\|u\|_{L^q(\R^N)}^p\le \frac{C}{N}\sum_{j=1}^N\|g_j\|_{L^1(\R^{N-1})}=\frac{C}{N}\sum_{j=1}^N\left(\int_{\Om}\left(\left|\frac{\partial u}{\partial x_j}(x)\right|^p-\left(\frac{p-1}{p}\right)^p\frac{|u(x)|^p}{d_j(x)^p}\right)\,dx\right)
\end{equation}
In the final step, we average \eqref{ll} over all possible choices of coordinates. More precisely, we apply \eqref{ll} to the coordinates $Re_1,\dots, Re_N$, $R\in\SO(N)$, integrate against the Haar measure on $\SO(N)$ and use 
\begin{align*}
\int_{\SO(N)}|Re_j\cdot\xi|^p\,dR&=\kappa_{N,p}|\xi|^p,\\
\int_{\SO(N)}d_{Re_j}(x)^{-p}\,dR&=\frac{1}{\left|\mathbb{S}^{N-1}\right|}\int_{\mathbb{S}^{N-1}}d_{e}(x)^{-p}\,d\sigma(e)=\kappa_{N,p}D_{\Om,p}(x)^{-p},
\end{align*}
where $\kappa_{N,p}=\frac{\sqrt{\pi}\,\Gamma\left(\frac{N+p}{2}\right)}{\Gamma\left(\frac{p+1}{2}\right)\Gamma\left(\frac{N}{2}\right)}$. That ends the proof. 
\end{proof}

\bigskip

\textbf{Acknowledgement:} A. Banerjee is supported by the Doctoral Fellowship of the Indian Statistical Institute, Delhi Centre. V. Sahu is supported by the Anusandhan National Research Foundation (ANRF) through the National Postdoctoral Fellowship (PDF/2025/004611).

\end{document}